\documentclass[12pt]{amsart}
\usepackage{amssymb,amsmath,epsf}
\usepackage{graphicx}
\usepackage{t1enc}
\usepackage[latin1]{inputenc}
\usepackage[german,english]{babel}
\usepackage{amsfonts}
\usepackage[all]{xy}
\usepackage{color}

\usepackage{geometry}
\newtheorem{defi}{Definition}[section]
\newtheorem{prop}[defi]{Proposition}
\newtheorem{theo}[defi]{Theorem}

\newtheorem{lem}[defi]{Lemma}

\newtheorem{coro}[defi]{Corollary}

\theoremstyle{remark}
\newtheorem{rem}[defi]{Remark}

\begin{document}
\title[Double Centralizers of Parabol. Subgroups of Braid Groups]{Double Centralizers of Parabolic Subgroups of Braid Groups}

\author[Garber, Kalka, Liberman, Teicher]{David Garber, Arkadius Kalka, Eran Liberman and Mina Teicher}

\address{David Garber, Department of Applied Mathematics, Faculty of
  Sciences, Holon Institute of Technology, 52 Golomb st., PO
  Box 305, 58102 Holon, Israel}
\email{garber@hit.ac.il}

\address{Arkadius Kalka and Eran Liberman, Department of Mathematics, Bar-Ilan University, 52900 Ramat-Gan, Israel}
\email{arkadius.kalka@rub.de, eranyliberman@gmail.com}

\address{Mina Teicher, Department of Mathematics, Bar-Ilan University, 52900 Ramat-Gan, Israel \\
and NYU School of Engineering}
\email{teicher@macs.biu.ac.il}

\begin{abstract}
We characterize the double centralizer of all parabolic subgroups of the braid groups.
We apply this result to provide a new and potentially more efficient solution to the subgroup conjugacy problem for parabolic
subgroups.
In the course of the proof we also characterize the centralizer for all parabolic subgroups.
\end{abstract}



\maketitle

\section{Introduction}
Fenn, Rolfsen and Zhu \cite{FRZ96} determined the centralizer of the standard parabolic subgroup $B_m$ in the braid group $B_n$ ($m<n$) \cite{Ar47}. This result was generalized by Paris \cite{Pa97} who computed generating sets for the centralizer of parabolic subgroups having connected associated Coxeter graphs for Artin groups of type $A$, $B$ and $D$. 

In this paper, we characterize the \emph{double centralizer} of any parabolic subgroup $H$ of $B_n$, namely we show that $C_{B_n}(C_{B_n}(H)) =Z(B_n) \cdot H$, where $Z(B_n)$ denotes the center of the braid group $B_n$.

Furthermore, we apply this result to the subgroup conjugacy problem for parabolic subgroups of $B_n$.
The conjugacy problem in the braid group $B_n$ was solved in the seminal paper of Garside \cite{Ga69}.
A more general problem is the {\it subgroup conjugacy problem for $H\le B_n$}: given two elements $x, y\in B_n$, and a subgroup $H\le B_n$, decide
whether $x$ and $y$ are conjugated by an element in $H$. In general, this problem is presumably undecidable, because $F_2\times F_2$ can be embedded in $B_n$ for $n\ge 5$ (where $F_2$ is the free group on two generators), and for $F_2\times F_2$, according to a result of Mihailova \cite{Mi58}, even the subgroup membership (or generalized word) problem is unsolvable.
Nevertheless, it is interesting to consider the subgroup conjugacy problem for particular natural subgroups of $B_n$. Indeed, even the
subgroup conjugacy problem for the natural embedded subgroups $B_m\le B_n$, for $m\le n$, has been open for $m\le n-2$ until \cite{KLT10}. The case $m=n-1$ was resolved in \cite{KLT09}, which was of particular interest, since the so-called {\it shifted conjugacy problem} \cite{De06}, which was also unknown to be solvable \cite{De06, LU08, LU09}, is equivalent to some subgroup conjugacy problem for $B_{n-1}$ in $B_n$ \cite{LU08,KLT09}.
In \cite{KLT09},  the subgroup conjugacy problem for $B_{n-1}\le B_n$ was transformed to an equivalent bi-simultaneous conjugacy problem.
Then, in \cite{KLT10}, the subgroup conjugacy problem for all parabolic subgroups of braid groups, even for all so-called Garside subgroups \cite{Go07} of Garside groups, was solved completely, and deterministic algorithms were provided. The solution in \cite{KLT10} does not resort to a detour via a simultaneous
conjugacy problem,
but there we cannot apply any cycling and decycling operations (see \cite{EM94}). Therefore, the
invariant subsets of the conjugacy class are quite large. \par

In this paper, we provide a second solution of the subgroup conjugacy problem for all parabolic subgroups of $B_n$. 
This solution is a generalization of the approach developed in \cite{KLT09}, namely we reduce the problem to an instance of a simultaneous conjugacy
problem.
Also the invariant subsets of the simultaneous conjugacy class involved in
Lee and Lee's solution \cite{LL02} are relatively big.  
In \cite{KTV14}, we introduce new much smaller invariant subsets
of the simultaneous conjugacy class, namely the so-called {\em Lexicographic Super Summit Sets}.
Using these new improved invariant subsets, our new approach to the
subgroup conjugacy problem for parabolic subgroups of $B_n$, given in Corollary \ref{subCPHRed},
is expected to be more efficient than the direct solution (using fractional normal forms) from \cite{KLT10}.

\medskip 

Though the subgroup conjugacy problem for standard parabolic subgroups of the braid groups deserves interest on its own, a particular motivation comes from applications in cryptography. Indeed, Dehornoy \cite{De06} proposed an authentication scheme based on the shifted conjugacy problem. We remark that, by using generalized shifted conjugacy operations, it is straightforward to construct shifted conjugacy problems which can be reduced to some subgroup conjugacy problem for $B_m\le B_n$.
Furthermore, the Diffie-Hellman public key exchange based on the braid group, introduced by Ko et al. \cite{KL+00}, relies on the subgroup conjugacy problem for $B_m\le B_n$. Though Gebhardt broke this cryptosystem in \cite{Ge06} with $100 \%$ success rate, using his ultra summit sets, introduced in \cite{Ge05}, we
have to point out that he did not provide a general solution to the subgroup conjugacy problem for $B_m\le B_n$.

\begin{rem}
As a further application of our main theorem on double centralizers, we mention the first deterministic solution to the {\em double coset problem}
for parabolic subgroups of braid groups (see also \cite{KTT14}).
\end{rem}

\medskip

{\sc Outline.} 
We provide two proofs for our main result on double centralizers.
The first proof, given in Section \ref{Elementary approach}, is elementary and the result is shown only for parabolic
subgroups with connected associated Coxeter graph (as in \cite{Pa97}).
Despite its limitation we expose this elementary proof here in detail, because one may succeed in future to "algebraize" its techniques completely. Then it may also apply to other Artin groups, e.g. of type B or D. \\
More precisely, Section \ref{Elementary approach} contains the following subsections.
Section \ref{prem} deals with generating sets of the centralizer of $\Delta _r^2$ in $B_n$. We simplify Gurzo's generating sets in order to determine the algebraic structure of these centralizers.
For the convenience of the reader, we include in Section \ref{nMin2} a solution to the subgroup conjugacy problem for $B_{n-2}$ in $B_n$, where we introduce the basic ideas which will be used in greater generality in the
subsequent two subsections. Then, Section \ref{CCparaH} describes the main result concerning double centralizers for parabolic subgroups
(with connected associated Coxeter graph).
Finally, Section \ref{subCPpara} contains the application to the subgroup conjugacy problem for these subgroups. \\
Section \ref{Full proof} contains the full proof. 
Here we view the braid group as the mapping class group of the $n$-punctured disc, and we use the non-trivial fact that,
for any braid $\beta $, every $\gamma \in C_{B_n}(\beta )$ preserves the canonical reduction system of $\beta $.
For the convenience of the reader, we consider in section \ref{Short proof} first the case $H=B_r$ ($r<n$). Then in section \ref{General proof}
we prove the main result for any parabolic subgroup $H$ of $B_n$.
There, in the course of the proof we also characterize the centralizer for all parabolic subgroups which seems to have been done so far
only for parabolic subgroups with connected associated Coxeter graph (see \cite{Pa97}).

\section{Elementary approach} \label{Elementary approach}
\subsection{Gurzo's presentation} \label{prem}

We use the following definitions. Let $\partial : B_{\infty } \longrightarrow B_{\infty }$ be the injective shift homomorphism, defined by
$\sigma _i \mapsto \sigma _{i+1}$.
\begin{defi} (\cite[Definition I.4.6.]{De00}) 
For $n\ge 2$, define $\delta _n=\sigma _{n-1}\cdots \sigma _2\sigma _1$. For $p, q \ge 1$, we set:
\[ \tau _{p,q}=\delta _{p+1}\partial (\delta _{p+1})\cdots \partial ^{q-1}(\delta _{p+1}), \]
i.e. the strands $p+1, \ldots , p+q$ cross over the strands $1, \ldots , p$ (see Figure \ref{taupq}).
\end{defi}

\begin{figure}[!ht]
\epsfysize 2cm
\epsfbox{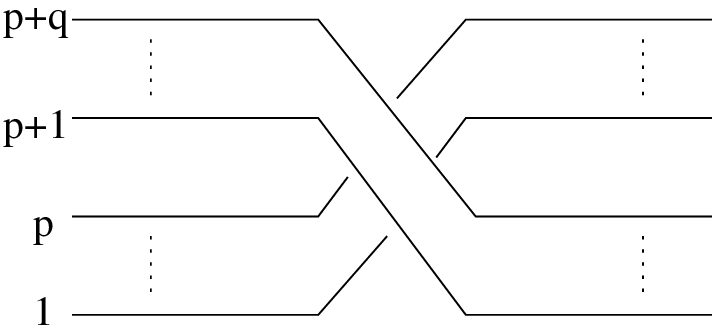}
\caption{$\tau _{p,q}$} \label{taupq}
\end{figure}



In particular, for $k\ge 0$ and $l \ge 1$ we denote (see Figures \ref{bk+1k+l1} and \ref{bk+11}):
\[ {\bar b}_{[k+1,k+l],1}=\tau _{k,l}\tau _{l,k} \quad  {\rm and} \quad  {\bar b}_{k+1,1}={\bar b}_{[k+1,k+1],1}=\tau _{k,1} \tau _{1,k}. \]

\begin{figure}[!ht]
\begin{minipage}{.5\textwidth}
  \centering
  \epsfysize 3cm
  \epsfbox{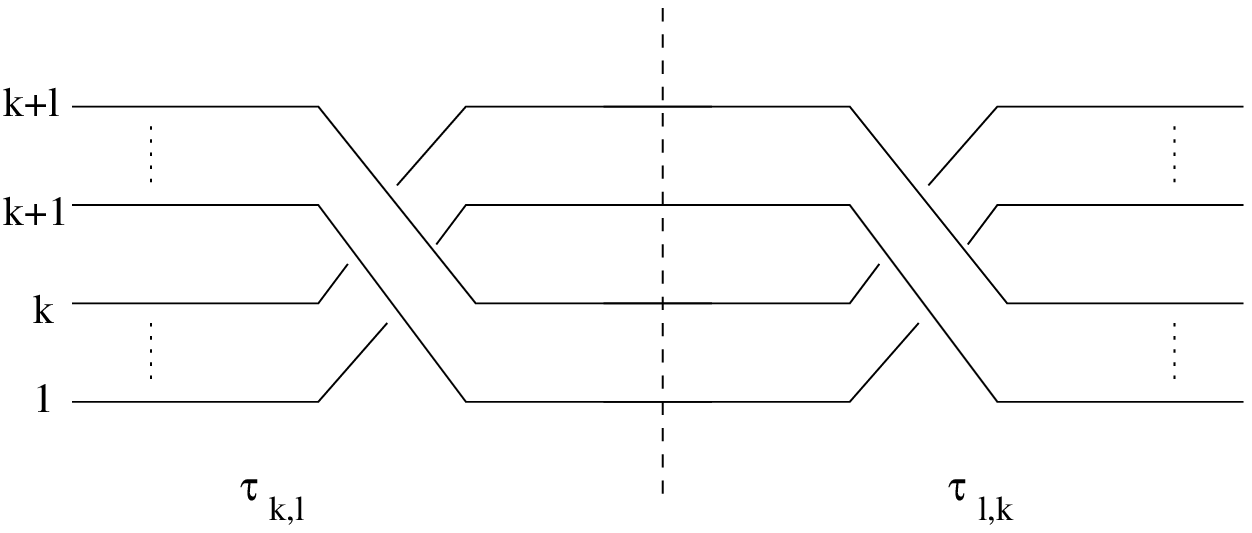}
  \caption{${\bar b}_{[k+1,k+l],1}$} \label{bk+1k+l1}
\end{minipage}%
\begin{minipage}{.5\textwidth}
  \centering
  \epsfysize 3cm
  \epsfbox{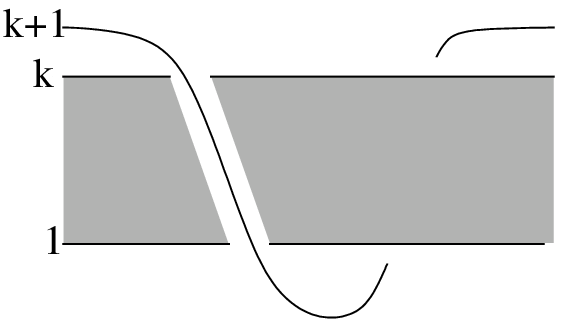}
  \caption{${\bar b}_{k+1,1}$} \label{bk+11}  
\end{minipage}
\end{figure}

According to Gurzo \cite{Gurzo85}, for $1\le r \le n-1$, the centralizer or $\Delta _r^2$ is given by:
\begin{prop}
$$C_{B_n}(\Delta_{r}^2) = B_r \cdot \langle \sigma_{r+1},\dots,\sigma_{n-1}, \bar{b}_{r+1,1},\bar{b}_{[r+1,r+2],1},\dots,\bar{b}_{[r+1,n],1}\rangle.$$
\end{prop}
Using Nielsen transformations we may simplify this generating set so that we obtain the complete algebraic structure of that centralizer and hence a presentation.

\begin{prop} \label{Gurzo}
For $1\le r \le n-1$, the centralizer $C_{B_n}(\Delta_{r}^2)$ is isomorphic to the direct product of the Artin groups of type $A_{r-1}$ and $B_{n-r}$.
In particular, we have:   
$$C_{B_n}(\Delta_{r}^2) = B_r \cdot \langle \bar{b}_{r+1,1}, \sigma_{r+1},\dots,\sigma_{n-1}\rangle \cong \mathcal{A}(A_{r-1}) \times \mathcal{A}(B_{n-r}).$$
\end{prop}

\begin{proof}
For all $2\le l \le n-r$, we have $\Delta _{r+l}^2=\bar{b}_{[r+1,r+l],1} \cdot \Delta _r^2 \cdot \partial ^r(\Delta _l^2)$ (see Figure \ref{br+1r+l1Delta2}).

\begin{figure}[!ht]
\epsfysize 4cm
\epsfbox{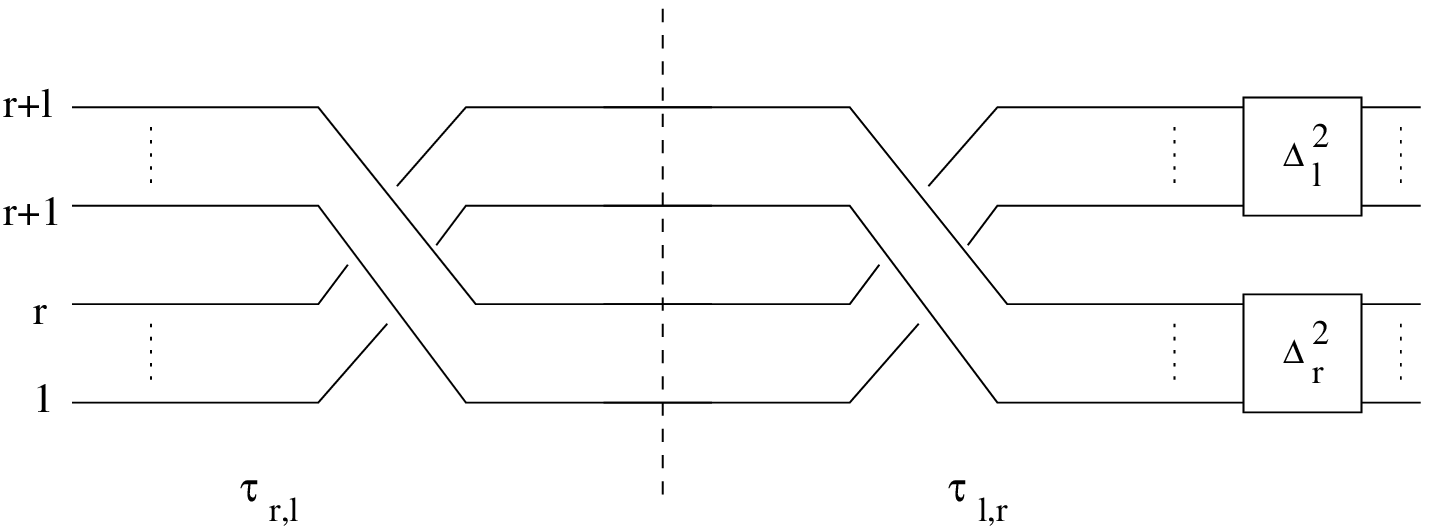}
\caption{Decompose $\Delta _{r+l}^2$ as $\bar{b}_{[r+1,r+l],1} \cdot \Delta _r^2 \cdot \partial ^r(\Delta _l^2)$.} \label{br+1r+l1Delta2}
\end{figure}

Since $\Delta _r^2 \in B_r$ and $\partial ^r(\Delta _l^2) \in \langle \sigma _{r+1}, \ldots , \sigma _{n-1} \rangle $, we may replace, for  $2 \le l \le n-r$,
the elements $\bar{b}_{[r+1,r+l],1}$ by the elements $\Delta _{r+l}^2$ in the generating set. Thus, we get:
$$C_{B_n}(\Delta_{r}^2) = B_r \cdot \langle \sigma_{r+1},\dots,\sigma_{n-1}, \bar{b}_{r+1,1}, \Delta _{r+2}^2, \ldots , \Delta _n^2 \rangle.$$
Starting with $\Delta _{r+1}^2=\Delta _r^2  \bar{b}_{r+1,1}$ (see Figure \ref{Delta2r+l}), 
we may prove by induction that, for $2 \le l \le n-r$,
\[ \Delta _{r+l}^2=\Delta _r^2  \bar{b}_{r+1,1} \bar{b}_{r+2,1} \cdots  \bar{b}_{r+l,1}, \]

\begin{figure}[!ht]
\begin{minipage}{.4\textwidth}
  \epsfysize 3cm
  \epsfbox{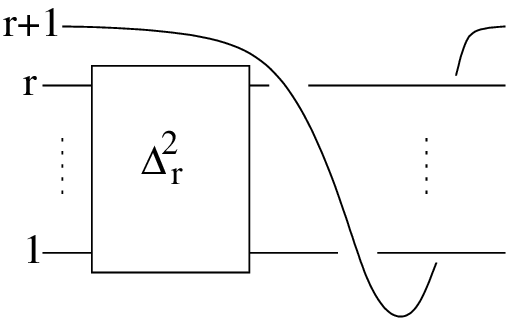}
\end{minipage}%
\begin{minipage}{.5\textwidth}
  \epsfysize 3cm
  \epsfbox{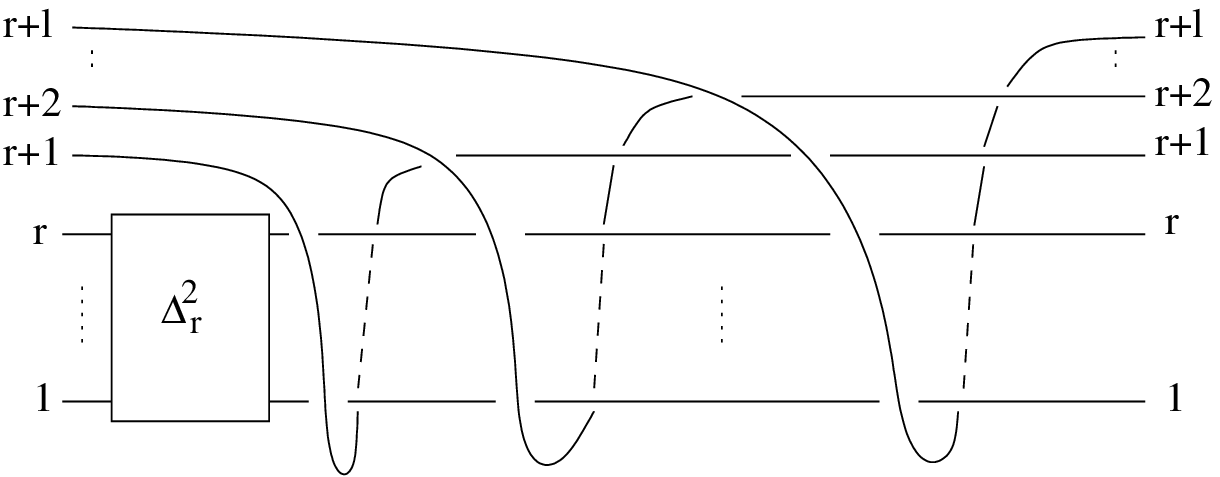}
\end{minipage}
  \caption{$\Delta _{r+1}^2=\Delta _r^2  \bar{b}_{r+1,1}$ \,\, and \,\, $\Delta _{r+l}^2=\Delta _r^2  \bar{b}_{r+1,1} \bar{b}_{r+2,1} \cdots  \bar{b}_{r+l,1}$.} \label{Delta2r+l}
\end{figure}

where all factors commute on the right hand side.
Furthermore, starting with $ \bar{b}_{r+2,1}=\sigma _{r+1} \cdot \bar{b}_{r+1,1} \cdot \sigma _{r+1}$, we may prove by induction that, for $2\le l \le n-r$ (see Figure \ref{br+l1}):
\[  \bar{b}_{r+l,1}= \sigma _{r+l-1} \cdots \sigma _{r+1} \cdot \bar{b}_{r+1,1} \cdot \sigma _{r+1} \cdots \sigma _{r+l-1}. \]

\begin{figure}[!ht]
  \epsfysize 4cm
  \epsfbox{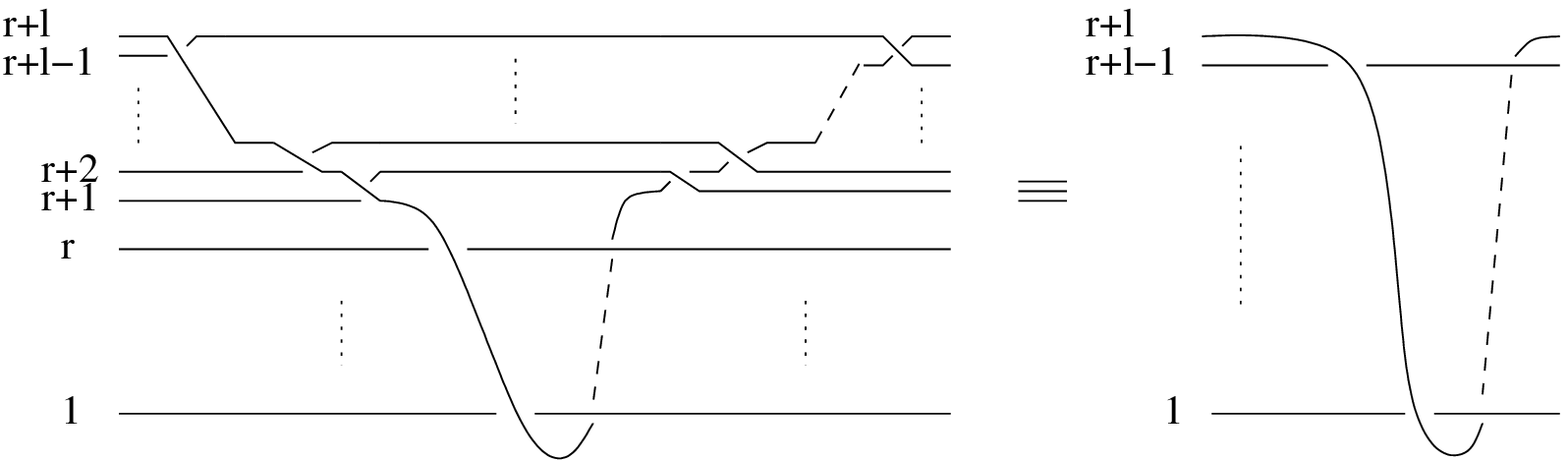}
  \caption{$\bar{b}_{r+l,1}= \sigma _{r+l-1} \cdots \sigma _{r+1} \cdot \bar{b}_{r+1,1} \cdot \sigma _{r+1} \cdots \sigma _{r+l-1}$} \label{br+l1}
\end{figure}

Thus, we may express, for all $2 \le l \le n-r$, the elements $\Delta _{r+l}^2$ as words over $ \bar{b}_{r+1,1}$, $\sigma _{r+1}, \ldots , \sigma _{n-1}$ only,
and hence, we may eliminate them from the generating set. Therefore, we have proven that:
\[C_{B_n}(\Delta_{r}^2) = B_r \cdot \langle \bar{b}_{r+1,1}, \sigma_{r+1},\dots,\sigma_{n-1}\rangle .\]
Consider the map $B_n \longrightarrow B_{n-r+1}$ which removes all but one of the strands $1, \ldots ,r$, say, all except for strand $1$.
This map is not a homomorphism, but the restriction $\eta : \langle  \bar{b}_{r+1,1}, \sigma_{r+1},\dots,\sigma_{n-1}\rangle  \longrightarrow B_{n-r+1}$
is an injective homomorphism with image $\langle \sigma _1^2, \sigma _2, \ldots , \sigma _{n-r} \rangle .$
Consider the Artin group $\mathcal{A}(B_{n-r})$ generated by $s_1, \ldots , s_{n-r}$ where $s_1$ and $s_2$ satisfy the 4-relation.
A standard embedding $\imath $ of this $B$-type Artin group into the braid group $B_{n-r+1}$ is given by $s_1 \mapsto \sigma _1^2$ and $s_i \mapsto \sigma _i$ for
$2 \le i \le n-r$. Hence $\eta ^{-1} \circ \imath : \mathcal{A}(B_{n-r}) \longrightarrow  \langle \bar{b}_{r+1,1}, \sigma_{r+1},\dots,\sigma_{n-1}\rangle $
is an isomorphism. Since $B_r \cong \mathcal{A}(A_{r-1})$ commutes with $\langle \bar{b}_{r+1,1}, \sigma_{r+1},\dots,\sigma_{n-1}\rangle $, we conclude that:
\[ C_{B_n}(\Delta _r^2) = B_r \cdot \eta ^{-1} \circ \imath (\mathcal{A}(B_{n-r})) \cong \mathcal{A}(A_{r-1}) \times \mathcal{A}(B_{n-r}). \]
\end{proof}

\begin{rem}
Since it it obvious which relations are fulfilled, we will call in the sequel $B_r \cdot \langle \bar{b}_{r+1,1}, \sigma_{r+1},\dots,\sigma_{n-1}\rangle $ \emph{Gurzo's presentation} of $C_{B_n}(\Delta_{r}^2)$.
\end{rem}

\subsection{The subgroup conjugacy problem for $B_{n-2}$ in $B_n$}   \label{nMin2}

We start with the following result concerning the centralizer of $\Delta_{n-2}^2$:

\begin{lem}\label{lemma1}
\[ B_{n-1} \cap C_{B_n}(\Delta_{n-2}^2) =B_{n-2} \cdot \langle \bar{b}_{n-1,1} \rangle = B_{n-2} \cdot \langle \Delta_{n-1}^2 \rangle. \]
\end{lem}

\begin{proof}
We start with the left equality. According to Gurzo's presentation (Proposition \ref{Gurzo} for $r=n-2$), we have: 
\[ C_{B_n}(\Delta_{n-2}^2) = B_{n-2} \cdot \langle \bar{b}_{n-1,1}, \sigma_{n-1} \rangle . \]
It suffices to show that $B_{n-1} \cap \langle \bar{b}_{n-1,1}, \sigma_{n-1} \rangle = \langle \bar{b}_{n-1,1} \rangle $.
Indeed, since $ \bar{b}_{n-1,1} \in B_{n-1}$ it suffices to show the inclusion 
$B_{n-1} \cap \langle \bar{b}_{n-1,1}, \sigma_{n-1} \rangle \subseteq \langle \bar{b}_{n-1,1} \rangle $.
Let $\beta \in B_{n-1} \cap \langle \bar{b}_{n-1,1}, \sigma_{n-1} \rangle $. Recall from the proof of Theorem \ref{Gurzo} (for $r=n-2$) the map $\eta $ which removes the strands $2, \ldots , n-2$. Hence, $\eta (\beta ) $ lies in
\[ \eta (B_{n-1} \cap \langle \bar{b}_{n-1,1}, \sigma_{n-1} \rangle ) \subseteq \eta (B_{n-1}) \cap \eta (\langle \bar{b}_{n-1,1}, \sigma_{n-1} \rangle )
= B_2 \cap \langle \sigma _1^2, \sigma_2 \rangle . \]
Now, $\eta (\beta ) \in B_2$ implies that there exists $k\in \mathbb{Z}$ such that $\eta (\beta )=\sigma _1^k$. 
Since $\eta (\beta )$ also lies in $\langle \sigma _1^2, \sigma_2 \rangle $ we may conclude that $k$ is even, i.e., $k=2k'$ for some $k' \in \mathbb{Z}$.
Recall that $\langle \sigma _1^2, \sigma_2 \rangle$ is the braid group on three strands which fixes the first strand, namely 
$\langle \alpha \in B_3 \mid \nu (\alpha )(1)=1 \rangle $ (see Figure \ref{sigma11,2}), 
where $\nu $ denotes the natural homomorphism which maps each braid to its induced permutation on the strands, i.e. $\nu : \sigma _i \mapsto (i,i+1)$. 
Therefore, we may view $\langle \sigma _1^2, \sigma_2 \rangle $ as the 2-strand braid group of the annulus \cite{Cr99}.
However, for odd $k$, we have $\nu (\sigma _1^k)(1)=2$, contradicting $\nu (\beta )(1)=1$. 

\begin{figure}[!ht]
  \epsfysize 2cm
  \epsfbox{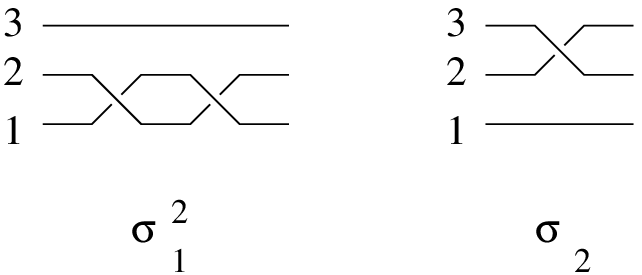}
  \caption{Generators of $\langle \alpha \in B_3 \mid \nu (\alpha )(1)=1 \rangle $} \label{sigma11,2}
\end{figure}

Thus, we have shown that $B_2 \cap \langle \sigma _1^2, \sigma_2 \rangle = \langle \sigma _1^2 \rangle $. For the braids in question, $\eta $ is an isomorphism.
Thus we may apply $\eta ^{-1}$, and we obtain 
\[  B_{n-1} \cap \langle \bar{b}_{n-1,1}, \sigma_{n-1} \rangle \subseteq \eta ^{-1}(\langle \sigma _1^2 \rangle )= \langle \bar{b}_{n-1,1} \rangle , \]
as needed. \par
The right equality follows from simple Nielsen transformations. 
Since $\Delta_{n-1}^2$ generates the center of $B_{n-1}$ and $\Delta_{n-2}^2 \in B_{n-1}$, we can write:
$$B_{n-1} \cap C_{B_n}(\Delta_{n-2}^2) \cong \langle B_{n-2}, \bar{b}_{n-1,1} \rangle \cong \langle B_{n-2}, \bar{b}_{n-1,1}, \Delta_{n-1}^2 \rangle.$$
Now, since $\bar{b}_{n-1,1}=\Delta_{n-1}^2 \Delta_{n-2}^{-2}$ and $\Delta_{n-2}^{-2}\in B_{n-2}$,
we can omit $\bar{b}_{n-1,1}$ from the last presentation, and hence we get:
$$B_{n-1} \cap C_{B_n}(\Delta_{n-2}^2) \cong \langle B_{n-2}, \bar{b}_{n-1,1}, \Delta_{n-1}^2 \rangle \cong \langle B_{n-2}, \Delta_{n-1}^2 \rangle.$$
\end{proof}

\begin{prop}
For all $x,y\in B_n$, the following are equivalent:
\begin{enumerate}
\item There exists $c \in B_{n-2}$ satisfying $y=c^{-1}xc$.
\item There exists $z \in B_n$ satisfying:
    \begin{itemize}
    \item[(a)] $y=z^{-1}xz$
    \item[(b)] $\Delta_{n-1}^2= z^{-1}\Delta_{n-1}^2 z$
    \item[(c)] $\Delta_{n-2}^2= z^{-1}\Delta_{n-2}^2 z$
    \item[(d)] $\sigma_{n-1} = z^{-1} \sigma_{n-1} z.$
    \end{itemize}
\end{enumerate}
\end{prop}

\begin{proof}
Since any element in $B_{n-2}$ commutes with $\Delta_{n-1}^2$, $\Delta_{n-2}^2$ and $\sigma_{n-1}$, the implication $(1) \Rightarrow (2)$
is obvious.

Due to Proposition 3 in \cite{KLT09}, Conditions (a) and (b) imply that $z=\Delta_n^{2p} c$ where $c \in B_{n-1}$. Hence, $c=\Delta_n^{-2p}z$.

Condition (c) implies that  $z \in C_{B_n}(\Delta_{n-2}^2)$, hence also $c \in C_{B_n}(\Delta_{n-2}^2)$.
Hence $c \in B_{n-1} \cap C_{B_n}(\Delta_{n-2}^2)$. By Lemma \ref{lemma1}, $c \in \langle B_{n-2}, \Delta_{n-1}^2 \rangle$, so we can write:
$c = \Delta_{n-1}^{2q}c'$ where $c' \in B_{n-2}$. For finishing the proof, we have to show that $q=0$.

We have: $z=\Delta_n^{2p}c =\Delta_n^{2p}\Delta_{n-1}^{2q}c'$. So $\Delta_{n-1}^{2q} =\Delta_n^{-2p}z \cdot (c')^{-1}$. Obviously, $\Delta_n^{-2p} \in C_{B_n}(\sigma_{n-1})$. By Condition (d), we have $z \in C_{B_n}(\sigma_{n-1})$, and by the construction $c' \in B_{n-2}$ and hence $c'\in C_{B_n}(\sigma_{n-1})$. Therefore, $\Delta_{n-1}^{2q} =\Delta_n^{-2p}z \cdot (c')^{-1} \in C_{B_n}(\sigma_{n-1})$. 
It is easy to proof by induction that the left greedy normal forms \cite{Th92, EM94} of $\Delta_{n-1}^{2q} \sigma_{n-1}$ and $\sigma_{n-1} \Delta_{n-1}^{2q}$ are
\[ \underbrace{\Delta_{n-1} \cdots \Delta_{n-1}}_{2q-1 \,\, {\rm factors}} \cdot (\Delta_{n-1}\sigma_{n-1}) \quad {\rm and} \quad
 (\sigma_{n-1} \Delta_{n-1}) \cdot \underbrace{\Delta_{n-1} \cdots \Delta_{n-1}}_{2q-1 \,\, {\rm factors}}, \]
respectively.
We conclude that  $\Delta_{n-1}^{2q} \not\in C_{B_n}(\sigma_{n-1})$ for $q\neq 0$, and since we have  $\Delta_{n-1}^{2q} \in C_{B_n}(\sigma_{n-1})$, it implies that $q=0$, as needed.
\end{proof}


\subsection{The double centralizer for a parabolic subgroup of $B_n$}   \label{CCparaH}

Now, we pass to the general case. 
We need the following result concerning the centralizer of $\Delta_{r}^2$:

\begin{lem}\label{claim}
The following holds for all $1 \le r \le n-1$:
\[ B_{r+1} \cap C_{B_n}(\Delta_r^2) =B_r \cdot \langle \bar{b}_{r+1,1} \rangle = B_r \cdot \langle \Delta_{r+1}^2 \rangle. \]
\end{lem}

\begin{proof}
The proof is a straightforward generalization of the proof of Lemma \ref{lemma1}.
Nevertheless, we provide full details for the convenience of the reader.
\medskip
We start with the left equality. According to Gurzo's presentation (Proposition \ref{Gurzo}), we have: 
\[ C_{B_n}(\Delta_r^2) = B_r \cdot \langle \bar{b}_{r+1,1}, \sigma_{r+1}, \ldots , \sigma _{n-1} \rangle . \]
It suffices to show that $B_{r+1} \cap \langle \bar{b}_{r+1,1}, \sigma_{r+1}, \ldots , \sigma_{n-1} \rangle = \langle \bar{b}_{r+1,1} \rangle $.
Indeed, since $ \bar{b}_{r+1,1} \in B_{r+1}$ it suffices to show the inclusion 
$B_{r+1} \cap \langle \bar{b}_{r+1,1}, \sigma_{r+1}, \ldots , \sigma_{n-1} \rangle \subseteq \langle \bar{b}_{r+1,1} \rangle $.
Let $\beta \in B_{r+1} \cap \langle \bar{b}_{r+1,1}, \sigma_{r+1}, \ldots , \sigma_{n-1} \rangle $. 
Consider the map $\eta $ which removes the strands $2, \ldots , r$. Hence, $\eta (\beta ) $ lies in 
\begin{eqnarray*} 
\eta (B_{r+1} \cap \langle \bar{b}_{r+1,1}, \sigma_{r+1}, \ldots , \sigma_{n-1} \rangle ) & \subseteq & 
\eta (B_{r+1}) \cap \eta (\langle \bar{b}_{r+1,1}, \sigma_{r+1}, \ldots , \sigma_{n-1} \rangle ) \\
&=& B_2 \cap \langle \sigma _1^2, \sigma_2, \ldots , \sigma_{n-r} \rangle . 
\end{eqnarray*}
Now, $\eta (\beta ) \in B_2$ implies that there exists $k\in \mathbb{Z}$ such that $\eta (\beta )=\sigma _1^k$. 
Since $\eta (\beta )$ also lies in $\langle \sigma _1^2, \sigma_2, \ldots , \sigma_{n-r} \rangle $, 
we may conclude that $k$ is even, i.e., $k=2k'$ for some $k' \in \mathbb{Z}$.
Recall, that $\langle \sigma _1^2, \sigma_2, \ldots , \sigma_{n-r} \rangle$ is the braid group on $n-r+1$ strands which fixes the first strand, namely 
$\langle \alpha \in B_{n-r+1} \mid \nu (\alpha )(1)=1 \rangle $, where $\nu $ denotes the natural homomorphism which maps each braid to its induced permutation on the strands, i.e. $\nu : \sigma _i \mapsto (i,i+1)$. Therefore, we may view $\langle \sigma _1^2, \sigma_2 , \ldots , \sigma_{n-r} \rangle $ as the $(n-r)$-strand braid group of the annulus \cite{Cr99}.
However, for odd $k$, we have $\nu (\sigma _1^k)(1)=2$, contradicting $\nu (\beta )(1)=1$. \par
Thus, we have shown that $B_2 \cap \langle \sigma _1^2, \sigma_2 , \ldots , \sigma_{n-r} \rangle = \langle \sigma _1^2 \rangle $. 
For the braids in question, $\eta $ is an isomorphism.
Thus we may apply $\eta ^{-1}$, and we obtain
\[  B_{r+1} \cap \langle \bar{b}_{r+1,1}, \sigma_{r+1}, \ldots , \sigma_{n-1} \rangle \subseteq \eta ^{-1}(\langle \sigma _1^2 \rangle )= \langle \bar{b}_{r+1,1} \rangle ,\]
as needed.
\par
The right equality follows from simple Nielsen transformations.
\end{proof}

\begin{rem}
 Given a group word over $\{ \bar{b}_{r+1,1}, \sigma_{r+1}, \ldots , \sigma_{n-1} \}$ representing an element in $B_{r+1}$, one may find a word over 
$\{ \bar{b}_{r+1,1} \}$ only, either by applying the map $\eta $ as explained above, or one computes the fractional (left) normal form \cite{Th92}
in the $B$-type Artin group $\langle  \bar{b}_{r+1,1}, \sigma_{r+1}, \ldots , \sigma_{n-1} \rangle $.
Fractional normal forms detect the standard parabolic subgroup in which an element of a finite type Artin group "lives".
This generalizes to Garside subgroups of Garside groups \cite{Go07}.
\end{rem}

Lemma \ref{claim} allows us to prove the following crucial result about centralizers:

\begin{prop}\label{nm}
For $1 \leq K \leq n-m$, we have:
$$ \bigcap_{k=1}^{K} \left[ C_{B_n}(\Delta_{n-k}^2) \cap C_{B_n}(\sigma_{n-k+1})\right]= \langle \Delta_n^2,B_{n-K} \rangle,$$
where we define $\sigma_n:=1$ and therefore $C_{B_n}(\sigma_n)=B_n$.
\end{prop}

\begin{proof}
We prove the theorem by induction on $K$. For $K=1$, according to Gurzo's presentation, we have:
\begin{eqnarray*}
C_{B_n}(\Delta_{n-1}^2) \cap C_{B_n}(\sigma_{n})& =& C_{B_n}(\Delta_{n-1}^2)=
\langle B_{n-1}, \bar{b}_{n,1} \rangle=\\
& =& \langle B_{n-1}, \bar{b}_{n,1} ,\Delta_n^2 \rangle=
\langle \Delta_n^2,B_{n-1} \rangle .
\end{eqnarray*}
The last equality holds since $\bar{b}_{n,1}=\Delta_n^2 \Delta_{n-1}^{-2} \in \langle \Delta_n^2,B_{n-1} \rangle$.

Now, assume that the equality holds for $K$ satisfying $1\leq K <n-m$, and we want to prove it for $K+1$.
We start by proving the following inclusion:
$$ \bigcap_{k=1}^{K+1} \left[ C_{B_n} \left( \Delta_{n-k}^2 \right) \cap C_{B_n}(\sigma_{n-k+1})\right] \subseteq \langle \Delta_n^2 \rangle \cdot B_{n-K-1} $$

We have:
\begin{eqnarray*}
& & \bigcap_{k=1}^{K+1} \left[ C_{B_n}(\Delta_{n-k}^2) \cap C_{B_n}(\sigma_{n-k+1})\right]=\\
& & = \left( \bigcap_{k=1}^{K} \left[ C_{B_n}(\Delta_{n-k}^2) \cap C_{B_n}(\sigma_{n-k+1})\right]\right) \cap C_{B_n}(\Delta_{n-K-1}^2) \cap C_{B_n}(\sigma_{n-K})=\\
& & = \langle \Delta_n^2,B_{n-K} \rangle \cap C_{B_n}(\Delta_{n-K-1}^2) \cap C_{B_n}(\sigma_{n-K}),
\end{eqnarray*}
where the last equality is by the induction hypothesis. Now, if\break
$z \in \bigcap_{k=1}^{K+1} \left[ C_{B_n}(\Delta_{n-k}^2) \cap C_{B_n}(\sigma_{n-k+1})\right]$, then $z \in \langle \Delta_n^2 \rangle \cdot B_{n-K}$, hence there exist $c \in B_{n-K}$ and $p\in \mathbb{Z}$
such that $z=\Delta_n^{2p} c$.

Since $z \in C_{B_n}(\Delta_{n-K-1}^2)$ and $z=\Delta_n^{2p} c$, then $c \in C_{B_n}(\Delta_{n-K-1}^2)$ too.
So $c \in B_{n-K} \cap C_{B_n}(\Delta_{n-K-1}^2) \stackrel{{\rm Lemma} \,\, \ref{claim}}{=}  B_{n-K-1} \cdot \langle \Delta_{n-K}^2 \rangle $.

Hence, there exist $c' \in B_{n-K-1}$ and $q \in \mathbb{Z}$ such that $c=\Delta_{n-K}^{2q} c'$. Therefore, $z=\Delta_n^{2p} c =\Delta_n^{2p} \Delta_{n-K}^{2q} c'$.
Equivalently: \[ \Delta_{n-K}^{2q}=\Delta_n^{-2p}z \cdot (c')^{-1}. \]

Recall again that $z \in  \langle \Delta_n^2,B_{n-K} \rangle \cap C_{B_n}(\Delta_{n-K-1}^2) \cap C_{B_n}(\sigma_{n-K})$, hence: $z \in  C_{B_n}(\sigma_{n-K})$. Also $c' \in B_{n-K-1}$, so: $c' \in  C_{B_n}(\sigma_{n-K})$. Obviously: $\Delta_n^{-2p} \in  C_{B_n}(\sigma_{n-K})$. Therefore: $\Delta_{n-K}^{2q} \in C_{B_n}(\sigma_{n-K})$, and hence $q=0$ (since $[\Delta_{n-K}^{i},\sigma_{n-K}]\neq 1$ for $i \neq 0$).

So we get $z =\Delta_n^{2p} c'$ where $c' \in B_{n-K-1}$. Therefore: $z \in \langle \Delta_n^2 \rangle \cdot B_{n-K-1}$. 
Since $z \in \bigcap_{k=1}^{K+1} \left[ C_{B_n}(\Delta_{n-k}^2) \cap C_{B_n}(\sigma_{n-k+1})\right]$, we get that:
$$ \bigcap_{k=1}^{K+1} \left[ C_{B_n}(\Delta_{n-k}^2) \cap C_{B_n}(\sigma_{n-k+1})\right] \subseteq \langle \Delta_n^2 \rangle \cdot B_{n-K-1}.$$

\medskip

The opposite inclusion is obvious, since every element of $B_{n-K-1}$ and $\Delta_n^2$ commute with $\Delta_{n-k}^2$ and $\sigma_{n-k+1}$ for $1 \leq k \leq K+1$. Therefore:
$$ \bigcap_{k=1}^{K+1} \left[ C_{B_n}(\Delta_{n-k}^2) \cap C_{B_n}(\sigma_{n-k+1})\right] = \langle \Delta_n^2 \rangle \cdot B_{n-K-1}.$$
This completes the induction step.
\end{proof}

\begin{theo} \label{CCBm}
For $1\le m <n$, we have:
\[ C_{B_n}(C_{B_n}(B_m))=\langle \Delta_n^2 \rangle \cdot B_m. \]
\end{theo}

\begin{proof}
According to \cite{FRZ96}, the centralizer of $B_m$ in $B_n$ is:
\[ C_{B_n}(B_m)=\langle \Delta _m^2, \Delta _{m+1}^2, \ldots , \Delta _{n-1}^2, \sigma _{m+1}, \ldots , \sigma _{n-1}  \rangle .  \]
We conclude that:
\[  C_{B_n}(C_{B_n}(B_m))=\bigcap_{k=1}^{n-m} \left[ C_{B_n}(\Delta_{n-k}^2) \cap C_{B_n}(\sigma_{n-k+1})\right] = \langle \Delta_n^2 \rangle \cdot B_m . \]
The right equality is by Proposition \ref{nm} (where we set $K=n-m$).
\end{proof}

We may extend this result to parabolic subgroups of $B_n$ with a connected associated Coxeter graph, in the following sense of Paris \cite{Pa97}.

\begin{defi}
A subgroup $H$ of the braid group $B_n$ is called {\em parabolic with a connected associated Coxeter graph} if
it is conjugate to $B_{[k,m]}=\langle \sigma _k, \sigma _{k+1}, \ldots , \sigma _{m-1} \rangle $ for some $1\le k<m \le n$.
\end{defi}

\begin{theo}  \label{CCH}
Let $H$ be parabolic subgroup of $B_n$ with a connected associated Coxeter graph such that $\gamma ^{-1}H\gamma =B_{[k,m]}$ for some
$\gamma \in B_n$ and $1\le k < m \le n$. Then the double centralizer of $H$ is given by:
\[ C_{B_n}(C_{B_n}(H))=\langle \Delta _n^2 \rangle \cdot H. \]
\end{theo}

\begin{proof}
Recall that $\tau _{m-k+1,k-1}$ is the braid satisfying that the strands $m-k+2, \ldots , m$ cross over the strands $1, \ldots , m-k+1$ (see Figure \ref{taupq}).
Therefore, $$B_{[k,m]}=\tau _{m-k+1,k-1}B_{m-k+1} \tau _{m-k+1,k-1}^{-1},$$ 
and we conclude that
$H=\gamma \tau _{m-k+1,k-1}B_{m-k+1} \tau _{m-k+1,k-1}^{-1} \gamma ^{-1}$.
Since $C_G(gHg^{-1})=gC_G(H)g^{-1}$ for any $g\in G$, for a group $G$ and $H\le G$, an application of Theorem \ref{CCBm} leads to the assertion.
\end{proof}

\subsection{The subgroup conjugacy problem for parabolic subgroups}  \label{subCPpara}

We apply the results of the preceding section to reduce the subgroup conjugacy problem to an instance of the simultaneous conjugacy problem.

\begin{theo} \label{Reduce}
Let $G$ be a group and $H\le G$ such that $C_G(C_G(H)) =Z(G) \cdot H$ where $Z(G)$ denotes the center of $G$.
Furthermore, let $\{g_1, \ldots , g_l\}$ be a generating set of $C_G(H)$. Then, for $x, y \in G$, the following are equivalent:
\begin{enumerate}
\item There exists $c \in H$ satisfying $y=c^{-1}xc$.
\item There exists $c' \in G$ satisfying 
    \begin{itemize}
    \item[(a)] $y=c'^{-1}xc'$, and
    \item[(b)] $g_i= c'^{-1} g_i c'$ for all $1 \leq i \leq l$.
    \end{itemize}
\end{enumerate}
\end{theo}

\begin{proof}
$(1) \Rightarrow (2)$: Set $c'=c\in H \le G$. Then $c'$ commutes with all elements in $C_G(H)$. \\
$(2) \Rightarrow (1)$: Conditions (b$_i$) implies that $c'$ commutes with all generators of   $C_G(H)$. Therefore, $z\in C_G(C_G(H))=Z(G) \cdot H$, and we may
write $c'=zc$ for some $z\in Z(G)$ and $c\in H$. From Condition (a), we conclude that: $y=c'^{-1}xc'=c^{-1}z^{-1}xzc=c^{-1}xc$.
\end{proof}

\begin{rem} 
In general, for any pair $(G,H)$, where $G$ is a group and $H \le G$, we have $C_G(C_G(H)) \supset Z(G) \cdot H$. An example of a proper inclusion is the pair $(F_n,F_m)$ for $m<n$, where we have $C_G(C_G(H)) =G$, since $C_G(H)=\{1\}$.
\end{rem}

Now, we may reduce the subgroup conjugacy problem for $B_m$ in $B_n$ (for $m<n$) to an instance of the simultaneous conjugacy problem:
\begin{coro}
Let $m<n$. For all $x,y\in B_n$, the following are equivalent:
\begin{enumerate}
\item There exists $c \in B_{n-m}$ satisfying $y=c^{-1}xc$.
\item There exists $z \in B_n$ satisfying
    \begin{itemize}
    \item[(a)] $y=z^{-1}xz$,
    \item[(b)] $\Delta_{n-i}^2= z^{-1}\Delta_{n-i}^2 z$  for all $1 \leq i \leq n-m$,
    \item[(c)] $\sigma_{n-i+1} = z^{-1} \sigma_{n-i+1} z,$  for all $2 \leq i \leq n-m$.
    \end{itemize}
\end{enumerate}
\end{coro}

\begin{proof}
The proof is just an application of Theorem \ref{Reduce} using the presentation
\[ C_{B_n}(B_m)=\langle \Delta _m^2, \Delta _{m+1}^2, \ldots , \Delta _{n-1}^2, \sigma _{m+1}, \ldots , \sigma _{n-1}  \rangle   \]
from \cite{FRZ96}. Also recall that $Z(B_n)=\langle \Delta _n^2 \rangle $ \cite{Ch48}.
\end{proof}

Slightly more general, we may apply that reduction to  parabolic subgroups of $B_n$ with a connected associated Coxeter graph:

\begin{coro} \label{subCPHRed}
Let $H$ be a parabolic subgroup of $B_n$ with a connected associated Coxeter graph such that $\gamma ^{-1}H\gamma =B_{[k,m]}$ for some
$\gamma \in B_n$ and $1\le k < m \le n$.
Then for all $x,y\in B_n$, the following are equivalent:
\begin{enumerate}
\item There exists $c \in H$ satisfying $y=c^{-1}xc$.
\item There exists $z \in B_n$ satisfying
    \begin{itemize}
    \item[(a)] $y=z^{-1}xz$,
    \item[(b)] {\tiny $\gamma \tau _{m-k+1,k-1}  \Delta_{n-i}^2 \tau _{m-k+1,k-1}^{-1} \gamma ^{-1}=
             z^{-1} (\gamma \tau _{m-k+1,k-1}  \Delta_{n-i}^2 \tau _{m-k+1,k-1}^{-1}  \gamma ^{-1}) z$},
             for all $1 \leq i \leq n-m+k-1$,
    \item[(c)] {\tiny $\gamma \tau _{m-k+1,k-1}  \sigma_{n-i+1} \tau _{m-k+1,k-1}^{-1}  \gamma ^{-1}=
                z^{-1} (\gamma \tau _{m-k+1,k-1}  \sigma_{n-i+1} \tau _{m-k+1,k-1}^{-1}) z$},
                for all $2 \leq i \leq n-m+k-1$.
    \end{itemize}
\end{enumerate}
\end{coro}

\begin{proof}
This is a straightforward consequence of Theorem \ref{Reduce} and Theorem \ref{CCH}.
\end{proof}

\begin{coro} \label{subCPH}
Let $H$ be a parabolic subgroup of $B_n$ with a connected associated Coxeter graph. Then the subgroup conjugacy problem for $H$ in $B_n$ is solvable.
\end{coro}

\begin{proof}
According to Corollary \ref{subCPHRed}, this problem may be reduced to an instance of the simultaneous conjugacy problem in $B_n$.
Now, the simultaneous conjugacy problem in braid groups was solved in \cite{LL02}. This implies the solvability of the subgroup conjugacy problem as well.
\end{proof}


\section{Short and general proof} \label{Full proof}
In this section we view the braid group $B_n$ as the mapping class group $\mathcal{MCG}_n$ of the $n$-punctured disc $D_n$.
This leads to a shortened proof that allows a simple generalization of the result to all parabolic subgroups of $B_n$.

\subsection{Short proof} \label{Short proof}
First we establish the result for $B_r \le B_n$.
We will need the following two lemmata. 
Recall the definition of the braid $\bar{b}_{r+1,1}$ from section \ref{prem} where strand $r+1$ goes around the first $r$ strands.

\begin{prop} \label{FRZ}
For $1\le r<n$ the centralizer of the standard parabolic subgroup $B_r$ in $B_n$ admits the following
presentation.
\[ C_{B_n}(B_r)=\langle \Delta _r^2, \bar{b}_{r+1,1}, \sigma _{r+1}, \ldots , \sigma _{n-1} \rangle . \]
\end{prop}

Using the embedding $\imath : \mathcal{A}(B_{n-r}) \longrightarrow B_n$ from section \ref{prem}, this result may be restated as
\[  C_{B_n}(B_r)=\langle \Delta _r^2 \rangle \cdot \imath(\mathcal{A}(B_{n-r})) \cong Z(B_n) \times \mathcal{A}(B_{n-r}). \]

\begin{proof}
The proof is a straghtforward simplification of a presentation
for $C_{B_n}(B_r)$ from Fenn, Rolfsen and Zhu \cite{FRZ96}. At the end of the chapter we give a short proof.
\end{proof}

The second lemma we need is a simple technical result on the centralizer in direct products.
\begin{lem} \label{productCentralizer}
Let $A, B$ be groups, and let $H$ be a subgroup of $A\times B$. Write $H=H_A \times H_B$, then
\[ C_{A\times B}(H)=C_A(H_A) \times C_B(H_B). \]
Furthermore, given an embedding $\imath : A\times B \longrightarrow G$, we have
\[ C_{\imath (A\times B)}(\imath (H_A\times H_B))=C_{\imath (A)}(\imath (H_A)) \cdot C_{\imath (B)}(\imath (H_B)). \]
\end{lem}

\begin{theo} \label{Referee}
For $1\le r<n$, the double centralizer of the standard parabolic subgroup $B_r$ in $B_n$ is given by 
\[ C_{B_n}(C_{B_n}(B_r))=\langle \Delta _r^2 \rangle \cdot B_r \cong Z(B_n) \times B_r. \]
\end{theo}

\begin{proof}
The following short proof is due to an anonymous referee. \\
Let $\mathcal{C}_r$ be the simple closed curve that encircles the first $r$ punctures in the disc $D_n$.
Viewing braids as mapping classes of $D_n$, the set of all braids that preserve $\mathcal{C}_n$ is exactly given by
\[  C_{B_n}(\Delta_r^2)=B_r \cdot \langle \bar{b}_{r+1,1}, \sigma _{r+1}, \ldots , \sigma _{n-1} \rangle
=B_r \cdot \imath(\mathcal{A}(B_{n-r})). \]
By Proposition \ref{FRZ}, we also have that 
\[ C_{B_n}(B_r) \le \{ \beta \in B_n \mid \beta \,\, {\rm preserves} \,\, \mathcal{C}_r \}. \]
The crucial ingredient of the proof is the fact that if $\beta \in B_n$ preserves $\mathcal{C}_r$, then also
every $\gamma \in C_{B_n}(\beta )$ preserves the curve $\mathcal{C}_r$ (see \cite{Iv92} or \cite{GW03}).
Hence, for $X=C_{B_n}(B_r)$, the centralizer of $X$ in $B_n$ equals the centralizer of $X$ in
$C_{B_n}(\Delta_r^2)=B_r \cdot \imath(\mathcal{A}(B_{n-r}))$. We conclude that
\[ C_{B_n}(C_{B_n}(B_r))=C_{B_r \cdot \imath(\mathcal{A}(B_{n-r}))}(C_{B_n}(B_r)), \]
and by Proposition \ref{FRZ} we get
\[ C_{B_n}(C_{B_n}(B_r))=C_{B_r \cdot \imath(\mathcal{A}(B_{n-r}))}(\langle \Delta _r^2 \rangle \cdot \imath(\mathcal{A}(B_{n-r}))). \]
Now, consider the following extension of the embedding $\imath $, namely the map
$\imath : B_r \times \mathcal{A}(B_{n-r}) \longrightarrow B_n$ defined by $\imath (B_r)=B_r$ and
$\imath(\mathcal{A}(B_{n-r}))= \langle \bar{b}_{r+1,1}, \sigma _{r+1}, \ldots , \sigma _{n-1} \rangle $ as above.
Hence, we may write 
\[ C_{B_n}(C_{B_n}(B_r))=C_{\imath (B_r \times \mathcal{A}(B_{n-r}))} (\imath(\langle \Delta _r^2 \rangle \times \mathcal{A}(B_{n-r})))), \]
which equals according to Lemma \ref{productCentralizer}
\[ C_{B_r}(\langle \Delta _r^2 \rangle ) \cdot Z(\imath(\mathcal{A}(B_{n-r})))=B_r \cdot 
\langle \Delta _r^{-2} \Delta _n^2 \rangle =B_r \cdot \langle \Delta _n^2 \rangle . \] 
\end{proof}

As in section \ref{CCparaH} this result extends to all parabolic subgroups of $B_n$ with connected associated Coxeter graph.

\begin{rem}
The set of all braids preserving the curve $\mathcal{C}_n$, namely the centralizer  $C_{B_n}(\Delta_r^2)$, has been shown in
\cite{Ro97, Go03} to coincide with the normalizer and the commensurator of $B_m$. Moreover, one has \cite{Ro97, Go03}
\[  C_{B_n}(\Delta_r^2)=N_{B_n}(B_m)={\rm Com}_{B_n}(B_m)=\langle B_m, C_{B_n}(B_m) \rangle =B_m \cdot QZ_{B_n}(B_m) \]
where $N$, Com, $QZ$ denote normalizer, commensurator and quasi-centralizer, respectively.
This result generalizes to all parabolic subgroups $H=A_X$ of an Artin system $(A, S)$ ($X \subseteq S$) of finite type, i.e.
we have (Thm. 0.1. in \cite{Go03})
\[ C_A(\Delta ^{\epsilon }_X)=N_A(H)={\rm Com}_A(H)=H \cdot QZ_A(H) \]
where $\epsilon \in \{1,2\}$ is minimal s.t. $\Delta ^{\epsilon }_X$ is central in $A_X$,
and $\langle H, C_A(H) \rangle $ is a normal subgroup of $N_A(H)$ such that $N_A(H)/ \langle H, C_A(H) \rangle $ is isomorphic
to the corresponding quotient given by replacing Artin groups by the corresponding Coxeter groups (see Thm. 0.3. in \cite{Go03}). 
Such results have been generalized further to Artin groups of type FC \cite{Go03b} with the only difference that there one
does not have Garside elements. \\
Quasi-centralizers of parabolic subgroups of Artin groups were characterized in terms of ribbons (Thm. 0.5. in \cite{Go03}), 
generalizing results from \cite{FRZ96}. For further generalizations like ribbon grupoids in Garside groups and ribbon categories
in Garside categories, see \cite{Go10} and \cite{DDGKM}, respectively.
\end{rem}

\begin{rem}
The same proof technique can be used to give a short proof of Proposition \ref{FRZ}. 
Recall that the embedding $\imath $ maps $B_r \times 1$ to $B_r$. Then we get
\[ C_{B_n}(B_r)=C_{\imath (B_r \times \mathcal{A}(B_{n-r}))}(\imath (B_r \times 1))=C_{B_r}(B_r)\cdot 
C_{\imath (\mathcal{A}(B_{n-r}))}(1)=Z(B_r)\cdot \imath (\mathcal{A}(B_{n-r})). \,\Box \]
\end{rem}

\subsection{General proof for all parabolic subgroups}  \label{General proof}
The techniques of the preceding subsection may be generalized to establish the Theorem for all parabolic subgroups of braid groups. Let us recall the defintion of a parabolic subgroup of a braid group, which actually
generalizes to all Artin groups. 

\begin{defi}
A subgroup $H$ of the $n$-strand braid group $B_n$ is called \emph{standard parabolic} if it is generated
by a non-empty subset of the standard Artin generators $\{\sigma _1, \ldots , \sigma _{n-1}\}$. \\
Furthermore, $H \le B_n$ is called \emph{parabolic} if it is conjugated to a standard parabolic subgroup of $B_n$.
\end{defi}

In the sequel, and without loss of generality (wlog), we assume that $H$ is isomorphic to $B_{r_1}\times \cdots \times B_{r_k}$ for some $2\le r_i$, $1\le i\le k$, with $r:=\sum _{i=1}^{k}\le n$. More precisely,
consider the embedding $\imath : B_{r_1}\times \cdots \times B_{r_k} \longrightarrow B_n$ defined by
$\imath (B_{r_i})=\partial ^{\sum _{j=1}^{i-1}r_j}(B_{r_i})$. Then we may assume wlog that
$H=\imath (B_{r_1}\times \cdots \times B_{r_k})$, i.e. $H=\prod _{i=1}^k \partial ^{\sum _{j=1}^{i-1}r_j}(B_{r_i})$. \\
We assume the reader is familiar with the concept of a canonical reduction system of a braid $\beta \in B_n \cong \mathcal{MCG}_n$.
For an explicit definition we refer to \cite{GW03}. Here we only introduce our notation which is similar to \cite{GW03}.

\begin{defi} \cite{GW03}
For $\beta \in B_n \cong \mathcal{MCG}_n$, denote by $R(\beta )$ the set of outermost curves in the
canonical reduction system of $\beta $. \\
For $H=\imath (X _{i=1}^k B_{r_i})$, $R(H)$ denotes the union $\bigcup _{i=1}^k \mathcal{C}_i$
where $\mathcal{C}_i$ is a simple closed curve enclosing the punctures $(\sum _{j=1}^{i-1} r_j)+1, \ldots , \sum _{j=1}^i r_j$.
\end{defi}


\begin{prop} \label{GW03prop} (see \cite{Iv92} or \cite{GW03})
For $\beta \in B_n$, every braid $\gamma \in C_{B_n}(\beta )$ preserves the canonical reduction system of $\beta $,
in particular its outermost part $R(\beta )$, i.e. let $B_{R(\beta )}$ be the set of braids that preserve 
$R(\beta )$, then 
\[ C_{B_n}(\beta ) \subseteq B_{R(\beta )}. \]
\end{prop}

Let us recall some elementary facts about centralizers which will prove useful in the sequel.

\begin{lem} \label{CHlem}
{\rm (1)} Let $H\subseteq A \subseteq B$. Then $C_A(H) \subseteq C_B(H)$. \\
{\rm (2)} Let $A, B \subseteq G$. Then $A \subseteq B$ if and only if $C_G(A) \supseteq C_G(B)$. \\
{\rm (3)} Let $H \subseteq A \subseteq G$ and $C_G(H) \subseteq A$. Then $C_G(H)=C_A(H)$.
\end{lem}

\begin{prop} \label{CHprop0}
For $H\le B_n$, let $B_{R(H)}$ denote the set of braids that preserve $R(H)$. Then 
\[ C_{B_n}(H)=C_{B_{R(H)}}(H). \]
\end{prop}

\begin{proof}
Consider a braid $\beta _0 \in H$ with $R(\beta _0)=R(H)$. 
Since $\{\beta _0\} \subset H$, Lemma \ref{CHlem} (2) implies $C_{B_n}(H)\subseteq C_{B_n}(\beta _0)$.   
And Proposition \ref{GW03prop} implies $C_{B_n}(\beta _0) \subseteq B_{R(H)}$. Hence, by transitivity and Lemma \ref{CHlem} (3), we get the assertion.
\end{proof}

For $H \le B_n$ standard parabolic, presentations for $B_{R(H)}$ have not been computed so far - except for the case that the associated Coxeter graph is connected.\\
However, $B_{R(H)}$ does not admit a nice direct product structure that allows us to separate $\imath (B_{r_1}\times \cdots \times B_{r_k})$
from the tubular part. The problem is that $B_{R(H)}$ also contains braids that may permute cycles $\mathcal{C}_i$ which enclose the same number $r_i$ of punctures.
Therefore, we consider a finite index subgroup of $B_{R(H)}$ which admits such a decomposition.

\begin{prop} \label{CHprop1}
For $H=\imath (B_{r_1}\times \cdots \times B_{r_k})$, let 
\[ \tilde{B}_{R(H)}:=\bigcap _{i=1}^k B_{R(\imath (B_{r_i}))} = \bigcap _{i=1}^k B_{\mathcal{C}_i} \le B_{R(H)}, i.e., \] 
$\tilde{B}_{R(H)}$ is the set of braids that preserve each cycle $\mathcal{C}_i$. Then 
\[ C_{B_n}(H)=C_{\tilde{B}_{R(H)}}(H). \]
\end{prop}

\begin{proof}
For $i=1, \ldots ,k$, consider braids $\beta _i \in \imath (B_{r_i})$ such that $R(\beta _i)=\mathcal{C}_i$.
Such braids can be realized by pseudo-Anosov braids on $r_i$ strands - shifted to the proper position. Hence we get
\[ C_{B_n}(H) \stackrel{\rm \ref{CHlem} (2)}{\subseteq } \bigcap_{i=1}^k C_{B_n}(\beta _i)
\stackrel{\ref{GW03prop}}{\subseteq } \bigcap_{i=1}^k B_{\mathcal{C}_i}=: \tilde{B}_{R(H)}. \]
The assertion follows from Proposition \ref{CHlem} (3). 
\end{proof}

\begin{rem} Note that the braid $\beta _0$ that appears in the proof of Proposition \ref{CHprop0} can be realized as the
product $\prod _{i=1}^k \beta _i$.
\end{rem}

In order to characterize $\tilde{B}_{R(H)}$ for our $H=\imath (B_{r_1}\times \cdots \times B_{r_k})$, we need the notion of braids that are "pure on some strands".

\begin{defi}
Let $X$ be a non-empty subset of $[n]:=\{1, \ldots ,n\}$. Then we define
\[ B_n(X):=\{ \beta \in B_n \mid \nu (\beta)(i)=i \quad \forall \,\, i\in X \}. \]
In particular, we denote $B_n(r):=B_n([r])$. 
\end{defi}

Note that $B_n(n)$ is the pure braid group $P_n$, and $B_n(1)$ is isomorphic to the Artin group 
$\mathcal{A}(B_{n-1})$. There exists a natural epimorphism of $B_n(r)$ onto $P_r$ by forgetting the $n-r$
punctures. For $1\le r<n$, the kernels of these epimorphisms are known as \emph{mixed braid group} (see, e.g.
\cite{Fr}). Lambropoulou computed presentations for mixed braid groups \cite{La02}. From these presentations one may
compute presentations for $B_n(r)$. But we won't need that in the sequel. \\

Consider the following extension of the embedding $\imath $, namely the map
\[ \imath : B_{r_1} \times \cdots \times B_{r_k} \times B_{n-r+k}(k) \longrightarrow B_n \]
where the first $k$ strands of $B_{n-r+k}(k)$ are mapped to cables of $r_i$ strands ($i=1, \ldots , k$) in $B_n$.
The range of that map, namely $\imath (B_{r_1} \times \cdots \times B_{r_k} \times B_{n-r+k}(k))$ is exactly $\tilde{B}_{R(H)}$.

\begin{theo} \label{CHprop2}
For $H$ (standard) parabolic as above, its centralizer admits the following algebraic structure:
\[ C_{B_n}(H) \cong Z(B_{r_1}) \times \cdots \times Z(B_{r_k}) \times B_{n-r+k}(k) \cong \mathbb{Z}^k \times B_{n-r+k}(k). \]
\end{theo} 

\begin{proof}
By Proposition \ref{CHprop1} we have
\[ C_{B_n}(H)=C_{\imath (B_{r_1} \times \cdots \times B_{r_k} \times B_{n-r+k}(k))} (\imath (B_{r_1} \times \cdots \times B_{r_k} \times 1). \]
Hence the direct product lemma implies 
\[ C_{B_n}(H)= (\prod _{i=1}^k \imath (Z(B_{r_i})) )\cdot \imath (B_{n-r+k}(k)). \]
\end{proof}
 
The centralizer of a parabolic subgroup of $B_n$ with a connected associated Coxeter graph is finitely generated. It was explicitly computed in \cite{FRZ96}.
This result was extended to  parabolic subgroups with connected associated Coxeter graph of Artin groups of type $B$ and $D$ in \cite{Pa97}. Here we may extend that result to all parabolic subgroups of $B_n$.

\begin{coro}
The centralizer of a parabolic subgroup of $B_n$ is finitely generated. 
\end{coro}

\begin{proof}
Recall the exact sequence $\ker \phi \longrightarrow B_n(r) \longrightarrow P_r $ mentioned above. The mixed braid group $\ker \phi$ is finitely generated by \cite{La02}. So is the pure braid group $P_r$, and we may conclude that $B_n(r)$ is finitely generated.
The assertion follows as corollary to Theorem \ref{CHprop2}.
\end{proof}

\begin{theo} Let $H$ be a parabolic subgroup of $B_n$. Then its double centralizer is given by
 \[ C_{B_n}(C_{B_n}(H))=H \cdot Z(B_n). \]
\end{theo}

\begin{proof}
It suffices to prove for $H$ standard parabolic as above. Since $C_{B_n}(H)= (\prod _{i=1}^k \imath (Z(B_{r_i})) )\cdot \imath (B_{n-r+k}(k))$, every
$\beta \in C_{B_n}(H)$ preserves each cycle $\mathcal{C}_i$ for all $i$. We conclude that $C_{B_n}(C_{B_n}(H)) \le \tilde{B}_{R(H)}$. 
Hence $C_{B_n}(C_{B_n}(H))=C_{\tilde{B}_{R(H)}}(C_{B_n}(H))$, i.e. we may compute the double centralizer as
\[ C_{\imath (B_{r_1} \times \cdots \times B_{r_k} \times B_{n-r+k}(k))}
(\imath (Z(B_{r_i}) \times \cdots \times Z(B_{r_k})\times B_{n-r+k}(k)) \]
by the direct product lemma. Hence we get
\begin{eqnarray*} 
C_{B_n}(C_{B_n}(H))&=&(\prod _{i=1}^k C_{\imath (B_{r_i})} (\langle \imath (\Delta _{r_i}^2) \rangle )) \cdot Z(\imath (B_{n-r+k}(k)))  \\
&=& (\prod _{i=1}^k \imath (B_{r_i})) \cdot \langle \Delta _n^2 \cdot (\prod _{i=1}^k \imath (\Delta _{r_i}^{-2})) \rangle = H \cdot \langle \Delta _n^2 \rangle . 
\end{eqnarray*}
\end{proof}

\begin{coro}
Let $H$ be a parabolic subgroup of $B_n$. The subgroup conjugacy problem for $H$ in $B_n$ is solvable.
\end{coro}

{\bf Acknowledgements.} We thank an anonymous referee for finding a gap in a previous version of the proof of the main result and for suggesting a shorter proof. 
We thank Boaz Tsaban for helpful discussions. \par
This work was partially supported by the Emmy Noether Research Institute for
Mathematics and the Minerva Foundation (Germany), the EU network ASSYAT, and
the Oswald Veblen Fund of the Institute for Advanced Study in Princeton.


\begin{thebibliography}{99}
\bibitem[Ar47]{Ar47} E. Artin, {\it Theory of braids}, Ann. Math. {\bf 48} (1947), 101--126.
\bibitem[Ch48]{Ch48} W.-L. Chow, {\it On the algebraic braid group}, Ann. Math. {\bf 49} (1948), 654--658.
\bibitem[Cr99]{Cr99} J. Crisp, {\it Injective maps between Artin groups}, Geom. Group Theory Down Under (Canberra
1996), de Gruyter, Berlin (1999), 119--137.
\bibitem[De00]{De00} P. Dehornoy, {\it Braids and Self-Distributivity}, Progress in Math. {\bf 192}, Birkhauser (2000).
\bibitem[De06]{De06} P. Dehornoy, {\it Using shifted conjugacy in braid-based cryptography}. In: L. Gerritzen, D. Goldfeld, M. Kreuzer, G. Rosenberger and V. Shpilrain (Eds.), {\it Algebraic Methods in Cryptography}, Contemp. Math. {\bf 418} (2006), 65--73.
\bibitem[DDGKM]{DDGKM} P. Dehornoy, F. Digne, E. Godelle, D. Krammer and J. Michel, {\it Foundations of Garside Theory},
 Europ. Math. Soc. Tracts in Mathematics, to appear (xiv + 689 pages). 
\bibitem[EC+92]{EC+92} D. B. A. Epstein, J. W. Cannon, D. F. Holt, S. V. F. Levy, M. S. Paterson and  W. P. Thurston, 
    {\it Word processing in groups}, Jones and Bartlett (1992).
\bibitem[EM94]{EM94} E. A. Elrifai and H.R. Morton, {\it Algorithms for positive braids}, Quart. J. Math. {\bf 45} (1994), 479--497.
\bibitem[FRZ96]{FRZ96} R. Fenn, D. Rolfsen and J. Zhu, {\it Centralizers in the braid group and singular braid monoid}, L'Enseignement Math. {\bf 42} (1996), 75--96.
\bibitem[Fr06]{Fr} Nuno Franco, {\it Conjugacy Problem for Subgroups with Applications to Artin Groups and Braid Type Group},
Communications in Algebra Volume {\bf 34}, Issue 11, (2006), 4207--4215.
\bibitem[Ga69]{Ga69} F.A. Garside, {\it The braid group and other groups}, Quart. J. Math. Oxford (2) {\bf 20} (1969), 235--254.
\bibitem[Ge05]{Ge05} V. Gebhardt, {\it A new approach to the conjugacy problem in Garside groups}, J. Alg. {\bf 292}(1) (2005), 282--302.
\bibitem[Ge06]{Ge06} V. Gebhardt, {\it Conjugacy search in braid groups}, Applicable Algebra in Engineering, Communication and Computing {\bf 17} (2006), 219--238.
\bibitem[Go03]{Go03} E. Godelle, {\it Normalisateur et groupe d'Artin de type sph\'{e}rique}, 
  Journal of Algebra {\bf 269} (1) (2003), 263--274.
\bibitem[Go03b]{Go03b} E. Godelle, {\it Parabolic subgroups of Artin groups of type FC}, 
  Pacific Journal of Math. {\bf 208} (2) (2003), 243--254.
\bibitem[Go07]{Go07} E. Godelle, {\it Parabolic subgroups of Garside groups}, J. Alg. {\bf 317} (2007), 1--16.
\bibitem[Go10]{Go10} E. Godelle, {\it Parabolic subgroups of Garside groups II: Ribbons},
Journal of Pure and Applied Algebra {\bf 214} (2010), 2044--2062.
\bibitem[Gu85]{Gurzo85} G.G. Gurzo, {\it Systems of generators for the normalizers of certain elements of the braid group}, Math. USSR Izvestiya {\bf 24}(3), 439--478 (1985).
\bibitem[GW04]{GW03} Juan Gonzales-Meneses and Bert Wiest, {\it On the structure of the centralizer of a braid}, Ann. Sci. Ec. Norm. Sup. {\bf 37} (2004), 729-757.
\bibitem[Iv92]{Iv92} N. V. Ivanov, {\it Subgroups of Teichm\"{u}ller modular groups}, Translations of mathematical
monographs vol. {\bf 115} (1992), AMS.
\bibitem[KL+00]{KL+00} K.H. Ko, S.J. Lee, J.H. Cheon, J.W. Han, J.-S. Kang and C. Park, {\it New public-key cryptosystem using braid groups},  Advances in Cryptology - CRYPTO 2000, LNCS {\bf 1880}, Springer (2000).
\bibitem[KLT09]{KLT09} A. Kalka, E. Liberman and M. Teicher, {\it A note on the shifted conjugacy problem in braid groups}, Groups - Complexity - Cryptology {\bf 1}(2) (2009), 227--230.
\bibitem[KLT10]{KLT10} A. Kalka, E. Liberman and M. Teicher, {\it Solution to the subgroup conjugacy problem for Garside subgroups of Garside groups},
  Groups - Complexity - Cryptology {\bf 2}(2) (2010), 157--174.
\bibitem[KTT14]{KTT14} A. Kalka, M. Teicher and B. Tsaban, {\it Double coset problem for parabolic subgroups of braid groups}, preprint:
{\tt http://arxiv.org/abs/1402.5541}
\bibitem[KTV14]{KTV14} A. Kalka, B. Tsaban and G. Vinokur, {\it Complete simultaneous conjugacy invariants in Garside groups}, preprint:
{\tt http://arxiv.org/abs/1403.4622}
\bibitem[La00]{La02} S. Lambropoulou, {\it Braid structures in knot complements, handlebodies and 3-manifolds}, in "Knots in Hellas '98'', C.McA. Gordon, V.F.R. Jones, L.H. Kauffman, S. Lambropoulou, J.H. Przytycki, Eds.; Series of Knots and Everything {\bf 24} World Scientific Press, 2000; pp. 274-289.
\bibitem[LL02]{LL02} S.J. Lee and E.K. Lee, {\it Potential weaknesses in the commutator key agreement protocol based on braid groups}, Advances  in Cryptology - EUROCRYPT 2002, LNCS {\bf 2332}, Springer (2002).
\bibitem[LU08]{LU08} J. Longrigg and A. Ushakov, {\it Cryptanalysis of shifted conjugacy authentication protocol}, J. Math. Crypto. {\bf 2} (2008), 107--114.
\bibitem[LU09]{LU09} J. Longrigg and A. Ushakov, {\it A practical attack on a certain braid group based shifted conjugacy authentication protocol}, Groups - Complexity - Cryptology {\bf 1}(2) (2009), 275--286.
\bibitem[Mi58]{Mi58} K.A. Mihailova, {\it The occurrence problem for direct products of groups}, Dokl. Akad. Nauk SSSR {\bf 119} (1958), 1103--1105. (Russian)
\bibitem[Pa97]{Pa97} L. Paris, {\it Centralizers of parabolic subgroups of Artin groups of type $A_l$, $B_l$ and $D_l$}, J. Alg. {\bf 196} (1997), 400--435.
\bibitem[Ro97]{Ro97} D. Rolfsen, {\it Braid subgroup normalisers, commensurators and induced representations},
Inv. math. {\bf 130} (3) (1997), 575--587.
\bibitem[Th92]{Th92} William Thurston, {\it Braid groups}, Chapter 9 in \cite{EC+92}.
\end{thebibliography}
\end{document}